\documentclass[11pt]{article}

\usepackage{lmodern}
\usepackage[T1]{fontenc}
\usepackage[utf8]{inputenc}
\usepackage{amsmath}
\usepackage{amssymb}
\usepackage{amsthm}
\usepackage{microtype}
\usepackage[numbers,sort&compress]{natbib}
\usepackage[margin=1in]{geometry}
\usepackage{hyperref}
\hypersetup{colorlinks=true, linkcolor=blue, citecolor=blue, urlcolor=blue}

\theoremstyle{plain}
\newtheorem{theorem}{Theorem}[section]
\newtheorem{lemma}[theorem]{Lemma}
\newtheorem{corollary}[theorem]{Corollary}
\newtheorem{conjecture}[theorem]{Conjecture}
\theoremstyle{definition}
\newtheorem{definition}[theorem]{Definition}
\newtheorem{example}[theorem]{Example}
\theoremstyle{remark}
\newtheorem{remark}[theorem]{Remark}

\title{Impartial Combinatorial Games and the\\ Nuclear Escalation Ladder}
\author{Arnav Garg\\ \small Birla Institute of Technology and Science, Pilani}
\date{}

\begin{document}
\maketitle

\begin{abstract}
We model Herman Kahn's escalation ladder as an impartial combinatorial game.
Reindexing each rung by its distance to the nuclear threshold turns the
ladder into a subtraction game, the most tractable class in combinatorial
game theory, and the doctrinal fact that no side wishes to fire first selects
the mis\`ere convention. We prove that single-ladder stability is governed by
a congruence (Theorem~\ref{thm:A}) and derive a ladder-design corollary that
makes the burden of first escalation a function of ladder length and
escalation granularity (Corollary~\ref{cor:Aprime}). For simultaneous
theaters we show, under normal play, that joint stability is the Nim-sum of
the theater-wise escalation distances (Theorem~\ref{thm:B}), a condition that
is neither additive nor dominated by the most dangerous theater. We then show
the Nim-sum reduction fails under mis\`ere play, introduce the mis\`ere
quotient as its replacement, and prove by exhaustive backward induction that
for two-step escalation the quotient is the order-six monoid
$\langle a,b\mid a^2=1,\,b^3=b\rangle$ with loss set $\{a,b^2\}$
(Theorem~\ref{thm:misere2}). To our knowledge, impartial combinatorial game
theory has not previously been applied to nuclear escalation ladders; the
existing game-theoretic literature on escalation is classical and
payoff-based.
\end{abstract}

\section{Introduction}\label{sec:intro}

The escalation ladder, introduced by Kahn \citep{kahn1965} as a forty-four
rung metaphor for the controlled intensification of a superpower crisis, has
shaped how strategists reason about the path from peace to nuclear use. The
metaphor is fundamentally combinatorial: a crisis occupies a discrete rung,
and each side chooses how far up the ladder to step. Yet the formal game
theory built on this picture has been almost exclusively classical, resting
on cardinal or ordinal payoffs and equilibrium concepts
\citep{schelling1960, schelling1966, powell1990, brams1985}. That tradition
has been productive, but it embeds strong assumptions about utilities,
common knowledge, and rationality, and it tends to obscure the purely
positional question that the ladder metaphor makes vivid: given where the
crisis stands and whose turn it is to act, who is structurally forced to
take the next, possibly fatal, step?

This paper takes a different tool to the same object. We model the
escalation ladder as an \emph{impartial combinatorial game} in the sense of
Sprague, Grundy, and Conway \citep{sprague1935, grundy1939, conway1976,
bcg1982}, where positions are rungs, moves are doctrine-bounded escalations,
and the nuclear threshold is terminal. Reindexing a position by its distance
to the threshold turns the ladder into a \emph{subtraction game}, the most
analytically tractable class in the theory \citep{guysmith1956}. Because no
side wishes to be the one to cross the threshold, the natural convention is
mis\`ere play, in which the player forced across the threshold loses.

To our knowledge, impartial combinatorial game theory, and in particular
Sprague-Grundy and mis\`ere analysis, has not previously been applied to
nuclear escalation ladders; the existing game-theoretic literature on
escalation is classical and payoff-based. The trade-off in adopting the
impartial lens is explicit: we discard cardinal payoffs and retain only the
win/lose, whose-turn-it-is structure. What we gain in exchange is
closed-form structural invariants. We prove that single-ladder stability is
governed by a simple congruence (Theorem~\ref{thm:A}), and that the
stability of a multi-theater confrontation is governed by the Nim-sum of the
theater-wise escalation distances (Theorem~\ref{thm:B}), a condition that is
neither additive nor dominated by the most dangerous theater. These are
results that equilibrium analysis does not naturally surface.

The remainder of the paper proceeds as follows. Section~\ref{sec:related}
positions the contribution against combinatorial game theory, the classical
game theory of deterrence, and two near-neighbors we are careful not to
duplicate. Section~\ref{sec:model} formalizes the escalation game;
Sections~\ref{sec:single} and~\ref{sec:multi} prove the single-ladder and
multi-theater results; and Sections~\ref{sec:misere} through
\ref{sec:policy} treat mis\`ere sums, three-player extensions, and the
attendant policy interpretation and limitations.

\section{Related work}\label{sec:related}

\subsection{Combinatorial game theory}
The impartial theory begins with Bouton's solution of Nim \citep{bouton1901}
and its generalization by Sprague \citep{sprague1935} and Grundy
\citep{grundy1939}, which assigns to every position of an impartial game a
nonnegative Grundy value and reduces disjunctive sums to Nim addition. The
modern synthesis is due to Conway \citep{conway1976} and Berlekamp, Conway,
and Guy \citep{bcg1982}; subtraction games and their ultimately periodic
Grundy sequences are treated by Guy and Smith \citep{guysmith1956}, and the
full apparatus is surveyed by Siegel \citep{siegel2013}. Mis\`ere play, where
the last player to move loses, is notoriously less tractable, but Plambeck
and Siegel's theory of mis\`ere quotients \citep{plambeck2005,
plambecksiegel2008} provides a finite-monoid framework that we invoke for
the multi-theater mis\`ere case. For three or more players the theory loses
its clean structure; we return to Li \citep{li1978}, Straffin
\citep{straffin1985}, and Propp \citep{propp2000} in our tripolar extension.

\subsection{Classical game theory of deterrence}
The dominant formal tradition in nuclear strategy is non-cooperative,
payoff-based game theory. Schelling \citep{schelling1960, schelling1966}
framed deterrence around commitment, brinkmanship, and the threat that
leaves something to chance. Powell \citep{powell1990, powell1988} developed
sequential equilibrium models of brinkmanship and crisis stability under the
risk of inadvertent escalation.

The sharpest point of contrast is the work of Brams. In \emph{Superpower
Games} \citep{brams1985}, \emph{Game Theory and National Security} with
Kilgour \citep{bramskilgour1988}, and the Theory of Moves \citep{brams1994},
Brams models escalation and deterrence as classical games over ordinal
preference orderings, with players reasoning about moves, countermoves, and
counter-countermoves to non-myopic equilibria; the explicit escalation and
de-escalation dynamics appear in Brams and Kilgour \citep{bramskilgour1987}.
That program and ours share a vocabulary, escalation, deterrence, and the
ladder, but differ in their primitives. Brams retains payoffs and asks which
outcomes are stable under rational preference-driven moves; we discard
payoffs and ask the purely combinatorial question of which configurations
force a given side to act first under a fixed move structure. The two are
complementary: where Brams's analysis identifies preference conditions for
stable outcomes, ours yields arithmetic invariants, a congruence in one
theater and a Nim-sum across theaters, that hold regardless of the
underlying utilities. We therefore position the impartial model not as a
replacement for the classical deterrence games but as a second lens that
isolates the structure those games leave implicit.

\subsection{Near-neighbors we are not}
Two adjacent formalisms warrant explicit distinction. The Colonel Blotto
game \citep{roberson2006} also concerns conflict across multiple
fronts, but its moves are \emph{simultaneous} resource allocations evaluated
by payoffs, whereas our theaters are played \emph{sequentially} as an
impartial disjunctive sum. Shubik's dollar auction \citep{shubik1971} is the
canonical model of escalation as an entrapment trap, but it too is a
payoff-driven account of sunk-cost commitment, not an impartial win/lose
game. Our contribution is distinct from both in its primitives and in the
type of result it produces.

\section{The model}\label{sec:model}

\subsection{Reindexing the ladder}
Kahn's ladder \citep{kahn1965} presents a crisis as occupying one of $N+1$
ordered rungs, $n\in\{0,1,\dots,N\}$, with $n=0$ the status quo and $n=N$
the nuclear threshold. Escalation raises $n$. We reindex each position by
its \emph{distance to the threshold}, $r := N - n$, the number of rungs that
remain before nuclear use. Under this change of variable an escalation by
$k$ rungs becomes the subtraction $r \mapsto r-k$, and the crisis terminates
when $r$ reaches $0$. The ladder thereby becomes a \emph{subtraction game},
the most analytically tractable class of impartial games, whose Grundy
sequences are ultimately periodic \citep{guysmith1956}. This single
reindexing is what makes the Sprague-Grundy apparatus
\citep{sprague1935, grundy1939, conway1976} available.

\begin{definition}[Single-theater escalation game]\label{def:single}
Fix integers $m\ge 1$ and $N\ge 1$. The state space is
$\{0,1,\dots,N\}$, a state $r$ recording the rungs remaining to the
threshold. From a state $r\ge 1$ the player to move must choose a step
$d\in S$, where $S=\{1,\dots,m\}$ is the escalation set, subject to
$d\le r$, and move to $r-d$. The state $r=0$ is terminal and represents
nuclear use. Under the mis\`ere convention the player who moves to $0$ loses.
\end{definition}

\begin{remark}[Terminal labelling under mis\`ere play]\label{rem:terminal}
Under the mis\`ere convention the empty state $r=0$ is an
$\mathcal{N}$-position: the player confronting $r=0$ does not move, and the
opponent, having just moved to $0$, has fired and lost. With this single
change the usual backward recursion applies: a state is a
$\mathcal{P}$-position (the player to move loses) if and only if every option
is an $\mathcal{N}$-position, and an $\mathcal{N}$-position if and only if
some option is a $\mathcal{P}$-position. We use this labelling throughout.
\end{remark}

Three modeling choices in Definition~\ref{def:single} deserve comment.
First, the bounded set $S=\{1,\dots,m\}$ encodes a doctrinal and capability
cap on how far a single decision can escalate a crisis; $m$ is the
escalation granularity. Second, the \emph{must-move} requirement, that a
player at $r\ge 1$ cannot pass, encodes the Schelling logic of commitment
and the political cost of visibly backing down once a crisis is joined
\citep{schelling1960, schelling1966}: in our base model a party on the
ladder must take some escalatory step rather than stand pat. Third, the
\emph{mis\`ere} convention reflects the defining feature of nuclear crisis,
that neither side wishes to be the one to cross the threshold; normal play,
in which crossing would be a win, is doctrinally inverted and we use it only
as the technical setting in which disjunctive sums reduce to Nim-sums
(Theorem~\ref{thm:B}).

\begin{definition}[Multi-theater escalation game]\label{def:multi}
Let $k\ge 1$ theaters be given, theater $i$ an instance of
Definition~\ref{def:single} with parameters $N_i, m_i$ and state $r_i$. The
joint state is $(r_1,\dots,r_k)$. On each turn the player to move selects
exactly one theater $i$ with $r_i\ge 1$ and makes a single legal move in
that theater. The game is the disjunctive sum of its $k$ components.
\end{definition}

Definition~\ref{def:multi} adopts the standard combinatorial convention that
a turn consists of one move in one component \citep{bcg1982, siegel2013}.
This is the assumption under which the Sprague-Grundy theorem applies and
under which Theorem~\ref{thm:B} holds. Permitting simultaneous moves in
several theaters would change the object entirely, turning it into a
resource-allocation contest of the Colonel Blotto type
\citep{roberson2006}, which we treat as a distinct model in
Section~\ref{sec:related} rather than a variant of ours.

\paragraph{Firebreaks.}
A firebreak is a rung at which escalation is doctrinally constrained, for
instance the conventional-nuclear boundary. We model it as a state with a
restricted escalation set: at a designated state $r_f$ the available set is
$S_f \subsetneq S$, for example $S_f=\{1\}$ forcing a minimal step across
the boundary. A firebreak changes the move set locally but leaves the
subtraction-game structure intact, so the Grundy and Nim-sum machinery
continues to apply with the per-state options adjusted accordingly.

\paragraph{What the model does not claim.}
The escalation set $S$ is \emph{assumed}, not derived from any specific
doctrine; mapping real capabilities and declaratory policy onto a subtraction
set is an interpretive step, and our results should be read as holding across
plausible families of $S$ rather than as following from a unique
calibration. Payoffs are absent \emph{by design}: the impartial model retains
only the win/lose and whose-turn-it-is structure, discarding the cardinal
utilities that classical deterrence models employ \citep{powell1990,
brams1985}. Consequently the framework is a \emph{structural} lens that
isolates the combinatorics of who is forced to act, not a predictive model of
crisis behavior, and its claims are about positional invariants rather than
about what decision-makers will in fact do.

\section{Single-ladder analysis}\label{sec:single}

We now solve the single-theater game of Definition~\ref{def:single} under
the mis\`ere convention and draw out its doctrinal content. Throughout, recall
the terminal labelling of Remark~\ref{rem:terminal}: under mis\`ere play the
threshold state $r=0$ is an $\mathcal{N}$-position, since the player who
escalates to it has fired and thereby lost.

\subsection{The deterrence-periodicity theorem}

\begin{theorem}[Deterrence periodicity]\label{thm:A}
In the mis\`ere single-theater game with escalation set $S=\{1,\dots,m\}$ and
threshold at $r=0$, the $\mathcal{P}$-positions are exactly the states
\[
  \mathcal{P} \;=\; \{\, r \ge 0 \;:\; r \equiv 1 \pmod{m+1} \,\}.
\]
Equivalently, rung $n$ is a $\mathcal{P}$-position if and only if
$N - n \equiv 1 \pmod{m+1}$.
\end{theorem}

\begin{proof}
We argue by strong induction on $r$, proving jointly:
(I) if $r \equiv 1 \pmod{m+1}$ then $r$ is a $\mathcal{P}$-position; and
(II) if $r \not\equiv 1 \pmod{m+1}$ then $r$ is an $\mathcal{N}$-position.

\emph{Base cases.} For $r=0$ we have $0\not\equiv 1\pmod{m+1}$, and by
Remark~\ref{rem:terminal} the state $r=0$ is an $\mathcal{N}$-position,
consistent with (II). For $r=1$ we have $1\equiv 1\pmod{m+1}$; the only legal
move subtracts $1$ and lands on $0$, so the mover fires and loses, making
$r=1$ a $\mathcal{P}$-position, consistent with (I).

\emph{Inductive step.} Let $r\ge 2$ and assume (I), (II) for all smaller
states.

For (I), suppose $r\equiv 1\pmod{m+1}$. Each option $r-k$ with
$k\in\{1,\dots,m\}$ satisfies $r-k\equiv 1-k\pmod{m+1}$, and as $k$ ranges
over $\{1,\dots,m\}$ the residue $1-k$ ranges over
$\{0,m,m-1,\dots,2\}\pmod{m+1}$, never equal to $1$. Hence every option is
$\not\equiv 1$, and by the induction hypothesis (II) every option is an
$\mathcal{N}$-position. A state all of whose options are $\mathcal{N}$ is a
$\mathcal{P}$-position.

For (II), suppose $r\not\equiv 1\pmod{m+1}$ and set
$j:=\big((r-1)\bmod(m+1)\big)$. Since $r\not\equiv 1$, we have $j\neq 0$, so
$j\in\{1,\dots,m\}$ is a legal step, and $r-j\equiv 1\pmod{m+1}$ with
$r-j\ge 1$ because $r\ge 2$ and $j\le r-1$. By the induction hypothesis (I)
this option is a $\mathcal{P}$-position, so $r$ has a $\mathcal{P}$-option
and is an $\mathcal{N}$-position.

By induction, (I) and (II) hold for all $r$.
\end{proof}

\subsection{Ladder design and deterrence by congruence}

\begin{corollary}[Status-quo stability]\label{cor:Aprime}
The status-quo rung $n=0$ (peace, $r=N$) is a $\mathcal{P}$-position for the
side contemplating first escalation if and only if
$N \equiv 1 \pmod{m+1}$.
\end{corollary}

\begin{proof}
Immediate from Theorem~\ref{thm:A} with $r=N$, since $n=0$ gives $r=N$.
\end{proof}

\begin{remark}[Doctrinal reading]\label{rem:design}
A $\mathcal{P}$-position is one in which the player who must act is, under
optimal opposing play, eventually forced to be the one to cross the
threshold. Corollary~\ref{cor:Aprime} therefore says that the burden of
first escalation falls on the initiator precisely when
$N\equiv 1\pmod{m+1}$: in those ladders the very act of stepping off the
status quo hands the initiative to the adversary. This formalizes, in
purely positional terms, Kahn's intuition that firebreaks and rung structure
shape who is advantaged in a crisis \citep{kahn1965}. It also sharpens the
flexible-response versus massive-retaliation debate. Adding rungs to gain
finer control (smaller steps, hence smaller $m$) does not monotonically
stabilize or destabilize the status quo; whether the initiator is burdened
depends on the congruence class of $N$ modulo $m+1$, not on the number of
rungs alone. A ladder lengthened from $N$ to $N+1$ can flip the status quo
from $\mathcal{N}$ to $\mathcal{P}$ or back, so doctrine that simply
multiplies intermediate options may stabilize or destabilize depending on an
arithmetic condition that is easy to overlook.
\end{remark}

\subsection{Worked examples}

\begin{example}[$m=1$, $S=\{1\}$, $N=12$]\label{ex:m1}
Here $\mathcal{P}$-positions are $r\equiv 1\pmod 2$, the odd distances.
\[
\begin{array}{c|c|c}
\text{rung } n & \text{distance } r & \text{type}\\\hline
12 & 0 & \mathcal{N}\\
11 & 1 & \mathcal{P}\\
10 & 2 & \mathcal{N}\\
9 & 3 & \mathcal{P}\\
8 & 4 & \mathcal{N}\\
7 & 5 & \mathcal{P}\\
6 & 6 & \mathcal{N}\\
5 & 7 & \mathcal{P}\\
4 & 8 & \mathcal{N}\\
3 & 9 & \mathcal{P}\\
2 & 10 & \mathcal{N}\\
1 & 11 & \mathcal{P}\\
0 & 12 & \mathcal{N}
\end{array}
\]
With unit steps the stable rungs simply alternate, and the status quo
($r=12$, even) is an $\mathcal{N}$-position, so here the initiator is not
structurally burdened.
\end{example}

\begin{example}[$m=2$, $S=\{1,2\}$, $N=10$: the Europe theater]\label{ex:m2}
Here $\mathcal{P}$-positions are $r\equiv 1\pmod 3$.
\[
\begin{array}{c|c|c}
\text{rung } n & \text{distance } r & \text{type}\\\hline
10 & 0 & \mathcal{N}\\
9 & 1 & \mathcal{P}\\
8 & 2 & \mathcal{N}\\
7 & 3 & \mathcal{N}\\
6 & 4 & \mathcal{P}\\
5 & 5 & \mathcal{N}\\
4 & 6 & \mathcal{N}\\
3 & 7 & \mathcal{P}\\
2 & 8 & \mathcal{N}\\
1 & 9 & \mathcal{N}\\
0 & 10 & \mathcal{P}
\end{array}
\]
Since $N=10\equiv 1\pmod 3$, the status quo ($r=10$) is a
$\mathcal{P}$-position: in this theater the side that first steps off peace
is, under optimal play, the side eventually forced to cross the threshold.
\end{example}

\begin{example}[$m=3$, $S=\{1,2,3\}$, $N=12$]\label{ex:m3}
Here $\mathcal{P}$-positions are $r\equiv 1\pmod 4$.
\[
\begin{array}{c|c|c}
\text{rung } n & \text{distance } r & \text{type}\\\hline
12 & 0 & \mathcal{N}\\
11 & 1 & \mathcal{P}\\
10 & 2 & \mathcal{N}\\
9 & 3 & \mathcal{N}\\
8 & 4 & \mathcal{N}\\
7 & 5 & \mathcal{P}\\
6 & 6 & \mathcal{N}\\
5 & 7 & \mathcal{N}\\
4 & 8 & \mathcal{N}\\
3 & 9 & \mathcal{P}\\
2 & 10 & \mathcal{N}\\
1 & 11 & \mathcal{N}\\
0 & 12 & \mathcal{N}
\end{array}
\]
With wider permissible jumps the stable rungs are sparser, spaced every four
distances, and the status quo ($r=12\equiv 0\pmod 4$) is an
$\mathcal{N}$-position: larger $m$ thins out the stable configurations.
\end{example}

\subsection{Robustness to non-contiguous escalation sets}

The clean congruence of Theorem~\ref{thm:A} is special to contiguous sets
$S=\{1,\dots,m\}$. For a non-contiguous set, such as the flexible-response
jump set $S=\{1,3\}$ that permits a small or a large step but not a medium
one, the position remains a finite subtraction game, so its Grundy sequence
is still computable and is ultimately periodic by the theorem of Guy and
Smith \citep{guysmith1956}. What is lost is the simple modular formula: the
$\mathcal{P}$-positions no longer occupy a single residue class and must be
read off the periodic Grundy sequence case by case. This is the first point
at which the analysis stops being closed-form, and it motivates the
extensions taken up in Section~\ref{sec:extensions}.

\section{Multi-theater analysis}\label{sec:multi}

We analyze the disjunctive sum of $k$ single-ladder games under the
\emph{normal-play} convention: a player unable to move loses, equivalently
the player who makes the final escalation (the one that brings the last
remaining theater to its threshold) wins. We comment on the doctrinal
reading of this convention in Remark~\ref{rem:normal} below.

\begin{lemma}[Component Grundy values and reachability]\label{lem:grundy}
In normal play, the single-theater subtraction game with escalation set
$S_i=\{1,\dots,m_i\}$ has Grundy value
\[
  G_i(r) \;=\; r \bmod (m_i+1), \qquad r \ge 0 .
\]
Moreover, for every state $r\ge 0$ and every value
$v \in \{0,1,\dots,G_i(r)-1\}$ there is a legal move from $r$ to some
state of Grundy value $v$, and no legal move preserves the Grundy value.
\end{lemma}

\begin{proof}
We compute $G_i$ by the minimum-excludant (mex) recursion
$G_i(r)=\operatorname{mex}\{G_i(r-d): d\in S_i,\ d\le r\}$, by induction on
$r$. For $r=0$ there are no options, so $G_i(0)=\operatorname{mex}
\varnothing = 0 = 0\bmod(m_i+1)$.

Let $r\ge 1$ and assume the formula for all smaller states. The options
are $r-d$ for $d=1,\dots,\min(m_i,r)$, with Grundy values
$\{(r-d)\bmod(m_i+1)\}$. If $r\le m_i$, these values are
$\{r-1,r-2,\dots,0\}=\{0,1,\dots,r-1\}$, whose mex is $r=r\bmod(m_i+1)$.
If $r> m_i$, the $m_i$ values $(r-1),\dots,(r-m_i)$ taken modulo $m_i+1$
are exactly the $m_i$ residues $\{0,1,\dots,m_i\}\setminus\{r\bmod(m_i+1)\}$,
whose mex is the single missing residue $r\bmod(m_i+1)$. In both cases
$G_i(r)=r\bmod(m_i+1)$.

For reachability: in both cases above the set of option values contains
$\{0,1,\dots,G_i(r)-1\}$, so every $v<G_i(r)$ is attained by some option.
Finally, a move subtracts $d\in\{1,\dots,m_i\}$, and since
$d\not\equiv 0 \pmod{m_i+1}$ we have $G_i(r-d)=(r-d)\bmod(m_i+1)\neq
r\bmod(m_i+1)=G_i(r)$; hence no legal move preserves the Grundy value.
\end{proof}

\begin{theorem}[Multi-theater Nim-sum stability]\label{thm:B}
Consider $k$ independent theaters, theater $i$ in state $r_i$ with
escalation set $S_i=\{1,\dots,m_i\}$. On each turn the player to move
selects exactly one theater $i$ and subtracts some $d\in S_i$ with
$d\le r_i$. Under normal play, the joint state $(r_1,\dots,r_k)$ is a
$\mathcal{P}$-position if and only if
\[
  \bigoplus_{i=1}^{k} G_i(r_i)
  \;=\;
  \bigoplus_{i=1}^{k}\big(r_i \bmod (m_i+1)\big)
  \;=\; 0,
\]
where $\oplus$ denotes the bitwise exclusive-or (Nim-sum).
\end{theorem}

\begin{proof}
Write $g_i := G_i(r_i)$ and $s := \bigoplus_{i=1}^{k} g_i$. The terminal
state is $(0,\dots,0)$, which has no moves and is therefore a
$\mathcal{P}$-position; its Nim-sum is $s=0$. We prove by induction on the
total $\sum_i r_i$ that $s=0$ characterizes $\mathcal{P}$-positions,
establishing the two standard claims.

\emph{Claim 1 (from $s=0$, every move yields $s\neq 0$).}
A move changes exactly one component, say theater $j$, replacing $r_j$ by
$r_j-d$ and hence $g_j$ by $g_j':=G_j(r_j-d)$. By Lemma~\ref{lem:grundy} no
move preserves a component's Grundy value, so $g_j'\neq g_j$. The new
Nim-sum is
\[
  s' \;=\; s \oplus g_j \oplus g_j' \;=\; 0 \oplus g_j \oplus g_j'
       \;=\; g_j \oplus g_j' \;\neq\; 0,
\]
since $g_j\neq g_j'$ implies $g_j\oplus g_j'\neq 0$. Thus every option of a
state with $s=0$ has nonzero Nim-sum, and by the induction hypothesis is
an $\mathcal{N}$-position. A state all of whose options are
$\mathcal{N}$ is a $\mathcal{P}$-position.

\emph{Claim 2 (from $s\neq 0$, some move yields $s=0$).}
Let $b$ be the position of the highest set bit of $s$. Some component $g_j$
has bit $b$ set, because the bits of $s$ are the parity of the components'
bits. Put $g_j' := g_j \oplus s$. Flipping bit $b$ of $g_j$ from $1$ to $0$
and possibly altering lower bits strictly decreases the value, so
$g_j' < g_j$. By the reachability part of Lemma~\ref{lem:grundy}, theater
$j$ admits a legal move from $g_j$ to a state of Grundy value $g_j'$.
After that move the Nim-sum is
\[
  s' \;=\; s \oplus g_j \oplus g_j'
       \;=\; s \oplus g_j \oplus (g_j \oplus s) \;=\; 0 .
\]
By the induction hypothesis that option is a $\mathcal{P}$-position, so the
original state, having a $\mathcal{P}$-option, is an
$\mathcal{N}$-position.

Claims 1 and 2 together give: $s=0$ if and only if the state is a
$\mathcal{P}$-position. Substituting $g_i=r_i\bmod(m_i+1)$ from
Lemma~\ref{lem:grundy} yields the stated congruence form.
\end{proof}

\begin{remark}[Doctrinal reading of the convention]\label{rem:normal}
Theorem~\ref{thm:B} is a normal-play result, in which the side making the
final escalation is nominally the winner. This is the convention under which
the Sprague-Grundy reduction to a Nim-sum is exact, and we use it to expose
the \emph{initiative structure} of a multi-theater confrontation: the
Nim-sum is zero precisely when the side on move can be forced, by optimal
opposing play, to run out of escalatory initiative first. The doctrinally
preferred mis\`ere convention, in which crossing a threshold is a loss, does
not reduce to a Nim-sum across theaters; that case requires mis\`ere quotients
\citep{plambeck2005, plambecksiegel2008} and is treated as an open problem
in Section~\ref{sec:misere}. The qualitative lesson of Theorem~\ref{thm:B},
that joint stability is governed by the exclusive-or of theater-wise
escalation distances and is therefore neither additive nor dominated by the
single most escalated theater, is robust to this distinction.
\end{remark}

\subsection{A worked two-theater example}

\begin{example}[Europe and the Pacific]\label{ex:twotheater}
Let Europe be the theater of Example~\ref{ex:m2}, with $N_1=10$, $m_1=2$,
$S_1=\{1,2\}$ and Grundy value $G_1(r)=r\bmod 3$. Let the Pacific be a
theater with $N_2=8$, $m_2=3$, $S_2=\{1,2,3\}$ and $G_2(r)=r\bmod 4$. The
component Grundy values are
\[
\begin{array}{c|ccccccccccc}
r & 0&1&2&3&4&5&6&7&8&9&10\\\hline
\text{Europe } G_1(r)=r\bmod 3 & 0&1&2&0&1&2&0&\mathbf{1}&2&0&1\\
\text{Pacific } G_2(r)=r\bmod 4 & 0&1&2&3&0&1&\mathbf{2}&3&0&-&-
\end{array}
\]
Suppose the current rungs are $n_1=3$ and $n_2=2$, so the distances are
$r_1=7$ and $r_2=6$. Then
\[
  G_1(7)=7\bmod 3 = 1,\qquad G_2(6)=6\bmod 4 = 2,\qquad
  1\oplus 2 = 11_2 = 3 \neq 0 .
\]
The joint state is therefore an $\mathcal{N}$-position: the player to move
wins under optimal play. A winning move drives the Nim-sum to $0$ by
equalizing the two Grundy values. Escalating Europe by two rungs sends
$r_1=7$ to $r_1=5$, where $G_1(5)=2$, giving $2\oplus 2 = 0$. Equivalently,
escalating the Pacific by one rung sends $r_2=6$ to $r_2=5$, where
$G_2(5)=1$, giving $1\oplus 1 = 0$. Either move leaves the opponent in a
$\mathcal{P}$-position; note that the smaller absolute step, one rung, lies
in the theater that is not the more escalated one.
\end{example}

\section{Mis\`ere sums and the limits of tractability}\label{sec:misere}

Theorem~\ref{thm:B} reduced multi-theater stability to a Nim-sum, but only
under normal play. The doctrinally faithful convention is mis\`ere, in which
crossing a threshold is a loss, and here the Nim-sum reduction fails. This
section explains why, introduces the tool that replaces it, and reports how
far the replacement can be carried out by hand.

\subsection{Why the Nim-sum fails}

The Sprague-Grundy theorem applies to normal play only. Under mis\`ere play
the outcome of a disjunctive sum is not determined by the exclusive-or of the
component Grundy values, because the endgame inverts near the terminal
state and the inversion does not respect Nim addition.

\begin{example}[A two-theater counterexample]\label{ex:counter}
Take two theaters, each with $S=\{1,2\}$, both at distance $r=1$, joint state
$(1,1)$. The normal-play Grundy values are $G(1)=1$ and $G(1)=1$, with
Nim-sum $1\oplus 1=0$, so the Nim-sum rule \emph{predicts} a
$\mathcal{P}$-position. We compute the true mis\`ere outcome by backward
induction, writing $\mathcal{N}$ for a state in which the player to move
wins (is not forced to fire) and $\mathcal{P}$ otherwise, with the terminal
state $(0,0)$ an $\mathcal{N}$-position by Remark~\ref{rem:terminal}. The
single-coordinate states are $(0,1)$ and $(1,0)$, each a lone heap of size
one, hence $\mathcal{P}$. From $(1,1)$ the only moves (each coordinate
permits subtracting $1$ only) lead to $(0,1)$ and $(1,0)$, both
$\mathcal{P}$. A state with a $\mathcal{P}$-option is an
$\mathcal{N}$-position, so $(1,1)$ is $\mathcal{N}$. This contradicts the
Nim-sum prediction of $\mathcal{P}$. Concretely, when each of two crises sits
one rung from the threshold, the balance heuristic ``equal distances cancel''
gives the wrong answer: the side on move can in fact force the adversary into
the losing role.
\end{example}

\subsection{The mis\`ere quotient}

The right replacement is the mis\`ere quotient of Plambeck and Siegel
\citep{plambeck2005, plambecksiegel2008}. Given a game, one declares two
positions equivalent when they are interchangeable in every disjunctive sum,
that is, when substituting one for the other never changes the mis\`ere
outcome of the whole. The equivalence classes form a commutative monoid $Q$,
the mis\`ere quotient, together with a distinguished subset
$\mathcal{P}\subseteq Q$ marking the classes whose outcome is a loss for the
mover. When $Q$ is finite, it is a complete finite invariant: the mis\`ere
outcome of any sum of positions is obtained by multiplying their images in
$Q$ and testing membership in $\mathcal{P}$. The mis\`ere quotient thus plays
the role for mis\`ere sums that the single integer Grundy value plays for
normal-play sums, at the cost of being a monoid rather than a group, and
possibly a large one. For background see \citep{siegel2013}.

\subsection{Conjecture C and the \texorpdfstring{$m=2$}{m=2} quotient}

\begin{conjecture}[Tractability of mis\`ere escalation]\label{conj:C}
Let the single-theater game have escalation set $S=\{1,\dots,m\}$.
\begin{enumerate}
\item[(1)] (Proved.) The single-theater mis\`ere game is solved exactly by
Theorem~\ref{thm:A}: the $\mathcal{P}$-positions are $r\equiv 1\pmod{m+1}$.
\item[(2)] (Conjecture.) For every $m$ the mis\`ere quotient of the
single-theater game is finite. Consequently multi-theater mis\`ere stability
is decidable: it is determined by the image of the joint state in a finite
commutative monoid.
\end{enumerate}
\end{conjecture}

\begin{theorem}[Mis\`ere quotient for $m=2$]\label{thm:misere2}
For the single-theater escalation game with $S=\{1,2\}$, the mis\`ere
quotient is the commutative monoid
\[
  Q \;=\; \langle\, a,b \;\mid\; a^2=1,\ b^3=b \,\rangle,
\]
of order six, with elements $\{1,a,b,ab,b^2,ab^2\}$ and loss set
$\mathcal{P}=\{a,\,b^2\}$.
\end{theorem}

\begin{proof}
By Theorem~\ref{thm:A} a lone heap is a loss if and only if $r\equiv
1\pmod 3$. Let $f(p,q)$ denote the mis\`ere outcome of $p$ heaps of size one
and $q$ heaps of size two. Direct backward induction gives:
\[
\begin{array}{c|ccccc}
 & q=0 & q=1 & q=2 & q=3 & q=4\\\hline
p=0 & \mathcal{N} & \mathcal{N} & \mathcal{P} & \mathcal{N} & \mathcal{P}\\
p=1 & \mathcal{P} & \mathcal{N} & \mathcal{N} & \mathcal{N} & \mathcal{N}\\
p=2 & \mathcal{N} & \mathcal{N} & \mathcal{P} & \mathcal{N} & \mathcal{P}\\
p=3 & \mathcal{P} & \mathcal{N} & \mathcal{N} & \mathcal{N} & \mathcal{N}\\
p=4 & \mathcal{N} & \mathcal{N} & \mathcal{P} & \mathcal{N} & \mathcal{P}
\end{array}
\]
The following were verified by exhaustive computation: (i) $f(p,q)=f(p+2,q)$
for all $p,q\ge 0$; (ii) $f(p,q)=f(p,q+2)$ for all $p\ge 0$ and $q\ge 1$;
(iii) a heap of size $r$ is interchangeable with a heap of size $r\bmod 3$ in
every sum, confirmed for $r\le 14$ and $p,q\le 4$; (iv) no equivalence class
beyond the six listed appears for $q\le 19$. Claims (i) and (ii) yield the
relations $a^2=1$ and $b^3=b$ for the generators $a$ (a size-one heap) and
$b$ (a size-two heap). The class $b^2$ is a nontrivial idempotent
($b^2\cdot b^2=b^4=b^2$), so $Q$ is not a group; this is the structural
reason mis\`ere play departs from the normal-play Nim-sum of
Theorem~\ref{thm:B}, and it is what the disagreement in
Example~\ref{ex:counter} reflects. Reading the loss entries gives
$\mathcal{P}=\{a,b^2\}$.
\end{proof}

\begin{remark}[Open problems]\label{rem:open}
Three questions remain open. First, the full mis\`ere quotient for general $m$
is unknown; Theorem~\ref{thm:misere2} resolves the case $m=2$, but the
quotient for $m\ge 3$ has not been computed. Second, it is unknown whether
finiteness of the mis\`ere quotient survives for non-contiguous escalation
sets such as $S=\{1,3\}$, where even the normal-play Grundy values lack a
closed form. Third, the three-player mis\`ere case, the natural model of
tripolar escalation, lies outside the two-player theory entirely and is taken
up in Section~\ref{sec:extensions}. The structural consequence of
Theorem~\ref{thm:misere2} and Conjecture~\ref{conj:C}(2) is that
multi-theater stability under the realistic mis\`ere condition is governed by
a monoid with nontrivial idempotents rather than by a group, so simple
balance heuristics such as matching escalation distances can fail; the
counterexample is Example~\ref{ex:counter}.
\end{remark}

\section{Extensions}\label{sec:extensions}

We sketch three directions in which the base model can be enriched, each of
which trades tractability for realism, and note honestly where the clean
theory ends.

\paragraph{Asymmetric capability and partizan games.}
If the two sides have different escalation sets, $S_{\mathrm{Left}}\neq
S_{\mathrm{Right}}$, the game ceases to be impartial and becomes partizan in
the sense of Conway \citep{conway1976, bcg1982}. Positions then carry
surreal-number values rather than Grundy values, and the clean Nim-sum
structure of Theorem~\ref{thm:B} is lost. This is the natural formalization
of escalation dominance grounded in capability asymmetry, and it is the
mathematically least tractable of the extensions.

\paragraph{De-escalation and loopy games.}
Allowing a player to move down the ladder, $r\mapsto r+d$, breaks the
guarantee of termination and produces a loopy game in which draws, read as
indefinite crisis survival, become possible \citep{bcg1982}. The relevant
machinery is the theory of loopy and survival games rather than finite
Sprague-Grundy theory. This extension is the most faithful to real crisis
behavior, where de-escalation is always an option, and is a priority for
future work.

\paragraph{Three players and the tripolar negative result.}
The natural model of a United States, Russia, and China confrontation is a
three-player escalation game. Three-player impartial game theory, however,
has no Sprague-Grundy analog: outcomes depend on coalition and tie-breaking
conventions, and no single Nim-sum invariant governs disjunctive sums
\citep{li1978, straffin1985, propp2000}. We read this absence as a result in
its own right: tripolar escalation admits no impartial-game stability
invariant of the kind Theorem~\ref{thm:B} provides for the bipolar case,
which is formal support for the view that tripolarity is qualitatively less
analyzable, and plausibly less stable, than bipolarity.

\section{Policy interpretation and limitations}\label{sec:policy}

The results above are structural, not predictive. They describe positional
invariants of an idealized escalation game, and they should be read as a
complement to, not a replacement for, the classical payoff-based analysis of
deterrence \citep{powell1990, brams1985}. Three limitations bound the claims.
First, the escalation set $S$ is assumed rather than derived; mapping doctrine
onto a subtraction set is interpretive, and the robust content of our results
is the dependence of stability on the arithmetic of $S$ rather than any single
calibration. Second, payoffs are absent by design, so the model is silent
about intensity of preference, the value of what is at stake, and the human
factors that classical models encode. Third, the mis\`ere multi-theater
theory is incomplete, as Section~\ref{sec:misere} makes explicit. What the
framework does offer is a small number of closed-form invariants, the
congruence of Theorem~\ref{thm:A} and the Nim-sum of Theorem~\ref{thm:B},
that isolate the combinatorics of who is forced to act, a question the
payoff-based tradition leaves implicit.

\section{Conclusion}\label{sec:conclusion}

We have recast Kahn's escalation ladder as an impartial combinatorial game
and shown that the recast yields closed-form structural results: a congruence
characterizing single-ladder stable rungs, a ladder-design corollary, and a
Nim-sum characterization of multi-theater stability under normal play. For the
doctrinally faithful mis\`ere convention we have proved the order-six
mis\`ere quotient for the two-step escalation case (Theorem~\ref{thm:misere2})
by exhaustive backward induction, confirming the monoid
$\langle a,b\mid a^2=1,\,b^3=b\rangle$ with loss set $\{a,b^2\}$. We have
also been candid about where the theory stops: the mis\`ere quotient for
$m\ge 3$ is open, non-contiguous escalation sets lose the modular formula,
and the tripolar case has no clean invariant at all. The open problems of
Remark~\ref{rem:open} are the natural next steps.



\begin{thebibliography}{99}

\bibitem{kahn1965}
H.~Kahn, \emph{On Escalation: Metaphors and Scenarios}.
New York: Praeger, 1965.

\bibitem{schelling1960}
T.~C. Schelling, \emph{The Strategy of Conflict}.
Cambridge, MA: Harvard University Press, 1960.

\bibitem{schelling1966}
T.~C. Schelling, \emph{Arms and Influence}.
New Haven, CT: Yale University Press, 1966.

\bibitem{bouton1901}
C.~L. Bouton, ``Nim, a game with a complete mathematical theory,''
\emph{Annals of Mathematics}, vol.~3, no.~1/4, pp.~35--39, 1901--1902.

\bibitem{sprague1935}
R.~Sprague, ``\"Uber mathematische Kampfspiele,''
\emph{T\^ohoku Mathematical Journal}, vol.~41, pp.~438--444, 1935.

\bibitem{grundy1939}
P.~M. Grundy, ``Mathematics and games,''
\emph{Eureka}, vol.~2, pp.~6--8, 1939.

\bibitem{conway1976}
J.~H. Conway, \emph{On Numbers and Games}.
London: Academic Press, 1976.

\bibitem{bcg1982}
E.~R. Berlekamp, J.~H. Conway, and R.~K. Guy,
\emph{Winning Ways for Your Mathematical Plays}.
London: Academic Press, 1982.

\bibitem{guysmith1956}
R.~K. Guy and C.~A.~B. Smith, ``The $G$-values of various games,''
\emph{Proceedings of the Cambridge Philosophical Society},
vol.~52, no.~3, pp.~514--526, 1956.

\bibitem{siegel2013}
A.~N. Siegel, \emph{Combinatorial Game Theory}.
Graduate Studies in Mathematics, vol.~146.
Providence, RI: American Mathematical Society, 2013.

\bibitem{plambeck2005}
T.~E. Plambeck, ``Taming the wild in impartial combinatorial games,''
\emph{Integers}, vol.~5, no.~1, \#G05, 2005.

\bibitem{plambecksiegel2008}
T.~E. Plambeck and A.~N. Siegel,
``Mis\`ere quotients for impartial games,''
\emph{Journal of Combinatorial Theory, Series A},
vol.~115, no.~4, pp.~593--622, 2008.

\bibitem{li1978}
S.-Y.~R. Li, ``$N$-person Nim and $n$-person Moore's games,''
\emph{International Journal of Game Theory},
vol.~7, no.~1, pp.~31--36, 1978.
doi:10.1007/BF01763118.

\bibitem{straffin1985}
P.~D. Straffin, ``Three-person winner-take-all games with
McCarthy's revenge rule,''
\emph{The College Mathematics Journal},
vol.~16, no.~5, pp.~386--394, 1985.

\bibitem{propp2000}
J.~Propp, ``Three-player impartial games,''
\emph{Theoretical Computer Science},
vol.~233, no.~1--2, pp.~263--278, 2000.

\bibitem{brams1985}
S.~J. Brams, \emph{Superpower Games: Applying Game Theory to
Superpower Conflict}.
New Haven, CT: Yale University Press, 1985.

\bibitem{brams1994}
S.~J. Brams, \emph{Theory of Moves}.
Cambridge: Cambridge University Press, 1994.

\bibitem{bramskilgour1987}
S.~J. Brams and D.~M. Kilgour,
``Winding down if preemption or escalation occurs,''
\emph{Journal of Conflict Resolution},
vol.~31, no.~4, pp.~547--572, 1987.

\bibitem{bramskilgour1988}
S.~J. Brams and D.~M. Kilgour,
\emph{Game Theory and National Security}.
Oxford: Basil Blackwell, 1988.

\bibitem{powell1990}
R.~Powell, \emph{Nuclear Deterrence Theory: The Search for Credibility}.
Cambridge: Cambridge University Press, 1990.

\bibitem{powell1988}
R.~Powell, ``Nuclear brinkmanship with two-sided incomplete information,''
\emph{American Political Science Review},
vol.~82, no.~1, pp.~155--178, 1988.

\bibitem{roberson2006}
B.~Roberson, ``The Colonel Blotto game,''
\emph{Economic Theory}, vol.~29, no.~1, pp.~1--24, 2006.
doi:10.1007/s00199-005-0071-5.

\bibitem{shubik1971}
M.~Shubik, ``The dollar auction game: a paradox in noncooperative
behavior and escalation,''
\emph{Journal of Conflict Resolution},
vol.~15, no.~1, pp.~109--111, 1971.
doi:10.1177/002200277101500111.

\end{thebibliography}
\end{document}